\documentclass{article}

\usepackage{amsmath}
\usepackage{amsthm}

\usepackage[pdftex]{hyperref}

\newtheorem{theorem}{Theorem}[section]

\numberwithin{equation}{section}

\title{Four New Generalized Fibonacci Number\\ Summation Identities}

\author{M.J. Kronenburg}
\date{}

\begin{document}

\maketitle

\begin{abstract}
Two new generalized Fibonacci number summation identities are stated and proved,
and two other new generalized Fibonacci number summation identities are derived from these,
of which two special cases are already known in literature.
\end{abstract}

\noindent
\textbf{Keywords}: Fibonacci number, Lucas number, generalized Fibonacci number.\\
\textbf{MSC 2010}: 11B39

\section{Four New Generalized Fibonacci Number\\ Summation Identities}

Let $F_n$ be the Fibonacci number, $L_n$ be the Lucas number
and $G_n$ be the generalized Fibonacci number, for which $G_{n+2}=G_{n+1}+G_n$ with
any seed $G_0$ and $G_1$ \cite{BQ03,CW00,V89}.
The two new generalized Fibonacci number summation identities are:
\begin{equation}\label{fibsum1}
 \sum_{k=0}^n \binom{n}{k} (-1)^k F_m^k F_{m+q}^{n-k} G_{p+qk} = (-1)^{nm} F_q^n G_{p-nm}
\end{equation}
\begin{equation}\label{fibsum2}
 \sum_{k=0}^n \binom{n}{k} (-1)^k L_m^k L_{m+q}^{n-k} G_{p+qk}
   = 5^{\lfloor n/2\rfloor} (-1)^{n(m+1)} F_q^n [ G_{p-nm+1} - (-1)^n G_{p-nm-1} ]
\end{equation}
where $\lfloor x\rfloor$ is the greatest integer less than or equal to $x$,
also called the floor of $x$ \cite{GKP94}.
Replacing $m$ by $-m$ and $q$ by $-q$ and using $F_{-n}=(-1)^{n+1}F_n$ and $L_{-n}=(-1)^nL_n$
these two identities transform into the following two identities:
\begin{equation}\label{fibsum3}
 \sum_{k=0}^n \binom{n}{k} (-1)^{(q+1)k} F_m^k F_{m+q}^{n-k} G_{p-qk} = F_q^n G_{p+nm}
\end{equation}
\begin{equation}\label{fibsum4}
 \sum_{k=0}^n \binom{n}{k} (-1)^{(q+1)k} L_m^k L_{m+q}^{n-k} G_{p-qk}
 = 5^{\lfloor n/2\rfloor} F_q^n [ G_{p+nm+1} - (-1)^n G_{p+nm-1} ] 
\end{equation}
The special cases $q=1$ and $q=-1$ of (\ref{fibsum3}) were already known in literature \cite{BQ03,CW00,V89}.

\section{Proof of the New Fibonacci Summation Identities}

The Fibonacci number summation identities (\ref{fibsum1}) and (\ref{fibsum2}) are proved for $G=F$
by first proving the $n=1$ case, then proving a recurrence relation
for the left side, and then showing that the right side fulfills
the same recurrence relation,
thus proving these identities by induction on $n$.
Then the identities are generalized to $G$ instead of $F$.
The following Binet formulas for $F_n$ and $L_n$ are well known \cite{BQ03,CW00,GKP94,V89}:
\begin{equation}
 \phi = \frac{1}{2} ( 1 + \sqrt{5} )
\end{equation}
\begin{equation}\label{fibodef}
 F_n = \frac{1}{\sqrt{5}} ( \phi^n - (1-\phi)^n )
\end{equation}
\begin{equation}\label{lucasdef}
 L_n = \phi^n + (1-\phi)^n
\end{equation}
First the following well known Fibonacci identity is needed \cite{BQ03,CW00,GKP94,V89}.
\begin{theorem}
\begin{equation}
 F_{m+1}F_n + F_mF_{n-1} = F_{n+m}
\end{equation}
\end{theorem}
\begin{proof}
Substituting formula (\ref{fibodef}) and using $\phi(1-\phi)=-1$ and $L_{n-1}+L_{n+1}=5F_n$:
\begin{equation}
\begin{split}
 & \frac{1}{5} [ ( \phi^{m+1}-(1-\phi)^{m+1} ) ( \phi^n-(1-\phi)^n ) \\
 & + ( \phi^m-(1-\phi)^m ) ( \phi^{n-1}-(1-\phi)^{n-1} ) ] \\
 = & \frac{1}{5} [ \phi^{n+m+1} + (1-\phi)^{n+m+1} - \phi^{m+1}(1-\phi)^n - (1-\phi)^{m+1}\phi^n \\
   & + \phi^{n+m-1} + (1-\phi)^{n+m-1} - \phi^m(1-\phi)^{n-1} - (1-\phi)^m\phi^{n-1} ] \\
 = & \frac{1}{5} [ L_{n+m+1} + L_{n+m-1} - (-1)^{m+1} L_{n-m-1} - (-1)^m L_{n-m-1} ] \\
 = & F_{n+m} \\
\end{split}
\end{equation}
\end{proof}
Substituting $-n$ for $n$ and using $F_{-n}=(-1)^{n+1}F_n$ this identity becomes:
\begin{equation}
 F_{n+1}F_m - F_nF_{m+1} = (-1)^n F_{m-n}
\end{equation}
Adding and subtracting these two identities and using $F_{n-1}+F_{n+1}=L_n$ and\\
$F_{n-1}-F_{n+1}=-F_n$ and $F_{m+1}-F_m=F_{m-1}$ and $F_{m-1}+F_{m+1}=L_m$ yields:
\begin{equation}
 F_{m+n}+(-1)^nF_{m-n} = L_nF_m
\end{equation}
\begin{equation}\label{lemmaf}
 F_{m+n}-(-1)^nF_{m-n} = F_n L_m
\end{equation}
Adding the identity with $m-1$ and $m+1$ and using
$F_{m-1}+F_{m+1}=L_m$ and $L_{m-1}+L_{m+1}=5F_m$ yields two similar identities:
\begin{equation}
 L_{m+n}+(-1)^nL_{m-n} = L_nL_m
\end{equation}
\begin{equation}\label{lemmal}
 L_{m+n}-(-1)^nL_{m-n} = 5 F_n F_m
\end{equation}
These last four identities were already known in literature \cite{BQ03,CW00,V89}.
Equations (\ref{lemmaf}) and (\ref{lemmal}) are used in the proof below.
\begin{theorem}
The following identity is the $n=1$ case of (\ref{fibsum1}):
\begin{equation}\label{case1f}
 F_{m+q}F_p - F_m F_{p+q} = (-1)^m F_q F_{p-m}
\end{equation}
\end{theorem}
\begin{proof}
Substituting formula (\ref{fibodef}) and using $\phi(1-\phi)=-1$ and (\ref{lemmal}) and\\
$F_{-n}=(-1)^{n+1}F_n$:
\begin{equation}
\begin{split}
   & \frac{1}{5} [ ( \phi^p-(1-\phi)^p ) ( \phi^{m+q}-(1-\phi)^{m+q} ) \\
   & - ( \phi^{p+q}-(1-\phi)^{p+q} ) ( \phi^m-(1-\phi)^m ) ] \\
 = & \frac{1}{5} [ \phi^{m+p+q} + (1-\phi)^{m+p+q} - \phi^p(1-\phi)^{m+q} - (1-\phi)^p\phi^{m+q} \\ 
   & - ( \phi^{m+p+q} + (1-\phi)^{m+p+q} - \phi^{p+q}(1-\phi)^m - (1-\phi)^{p+q}\phi^m ) ] \\
 = & -\frac{1}{5} [ (-1)^p L_{m-p+q} - (-1)^{p+q} L_{m-p-q} ] \\
 = & -\frac{1}{5} (-1)^p [ L_{m-p+q} - (-1)^q L_{m-p-q} ] \\
 = & -(-1)^p F_q F_{m-p} \\
 = & (-1)^m F_q F_{p-m} \\
\end{split}
\end{equation}
\end{proof}
\begin{theorem}
The following identity is the $n=1$ case of (\ref{fibsum2}):
\begin{equation}\label{case1l}
 L_{m+q}F_p - L_m F_{p+q} = (-1)^{m+1} F_q L_{p-m}
\end{equation}
\end{theorem}
\begin{proof}
Substituting formula (\ref{fibodef}) and (\ref{lucasdef}) and using $\phi(1-\phi)=-1$ and (\ref{lemmaf})
and $L_{-n}=(-1)^nL_n$:
\begin{equation}
\begin{split}
   & \frac{1}{\sqrt{5}} [ ( \phi^p-(1-\phi)^p ) ( \phi^{m+q}+(1-\phi)^{m+q} ) \\
   & - ( \phi^{p+q}-(1-\phi)^{p+q} ) ( \phi^m+(1-\phi)^m ) ] \\
 = & \frac{1}{\sqrt{5}} [ \phi^{m+p+q} - (1-\phi)^{m+p+q} + \phi^p(1-\phi)^{m+q} - (1-\phi)^p\phi^{m+q} \\ 
   & - ( \phi^{m+p+q} - (1-\phi)^{m+p+q} + \phi^{p+q}(1-\phi)^m - (1-\phi)^{p+q}\phi^m ) ] \\
 = & - [ (-1)^p F_{m-p+q} - (-1)^{p+q} F_{m-p-q} ] \\
 = & - (-1)^p [ F_{m-p+q} - (-1)^q F_{m-p-q} ] \\
 = & -(-1)^p F_q L_{m-p} \\
 = & (-1)^{m+1} F_q L_{p-m} \\
\end{split}
\end{equation}
\end{proof}
There are variants of (\ref{case1f}) and (\ref{case1l}) obtained by adding the
identity with $p-1$ and $p+1$ and using $F_{p-1}+F_{p+1}=L_p$ and $L_{p-1}+L_{p+1}=5F_p$:
\begin{equation}\label{case1fl}
 F_{m+q}L_p - F_m L_{p+q} = (-1)^m F_q L_{p-m}
\end{equation}
\begin{equation}\label{case1ll}
 L_{m+q}L_p - L_m L_{p+q} = 5 (-1)^{m+1} F_q F_{p-m}
\end{equation}

\begin{theorem}
Let $S(n,m,p,q)$ be the left side of identity (\ref{fibsum1}), then:
\begin{equation}
 S(n,m,p,q) = F_{m+q} S(n-1,m,p,q) - F_m S(n-1,m,p+q,q)
\end{equation}
\end{theorem}
\begin{proof}
\begin{equation}
\begin{split}
 & F_{m+q} \sum_{k=0}^{n-1} \binom{n-1}{k}(-1)^k F_m^k F_{m+q}^{n-k-1} F_{p+qk} \\
 & - F_m \sum_{k=0}^{n-1} \binom{n-1}{k}(-1)^k F_m^k F_{m+q}^{n-k-1} F_{p+q(k+1)} \\
 = & F_{m+q} \sum_{k=0}^{n-1} \binom{n-1}{k}(-1)^k F_m^k F_{m+q}^{n-k-1} F_{p+qk} \\
 & - F_m \sum_{k=1}^n \binom{n-1}{k-1}(-1)^{k-1} F_m^{k-1} F_{m+q}^{n-k} F_{p+qk} \\
 = & \sum_{k=0}^{n-1} \binom{n-1}{k}(-1)^k F_m^k F_{m+q}^{n-k} F_{p+qk} \\
 & + \sum_{k=1}^n \binom{n-1}{k-1}(-1)^k F_m^k F_{m+q}^{n-k} F_{p+qk} \\
\end{split}
\end{equation}
For $0<k<n$ the binomial coefficient addition formula \cite{GKP94} is used:
\begin{equation}
 \binom{n-1}{k} + \binom{n-1}{k-1} = \binom{n}{k}
\end{equation}
and for $k=0$ and $k=n$ separately the theorem is proved.
\end{proof}
When $S(n,m,p,q)$ is the left side of identity (\ref{fibsum2}) the theorem is:
\begin{equation}
 S(n,m,p,q) = L_{m+q} S(n-1,m,p,q) - L_m S(n-1,m,p+q,q)
\end{equation}
and the proof is similar.
When $S(n,m,p,q)$ is the right side of identities (\ref{fibsum1}) and (\ref{fibsum2}) the
same recurrence relations are fulfilled. 
\begin{theorem}
 Let $S(n,m,p,q)$ be the right side of (\ref{fibsum1}), then also:
\begin{equation}
 S(n,m,p,q) = F_{m+q} S(n-1,m,p,q) - F_m S(n-1,m,p+q,q)
\end{equation}
\end{theorem}
\begin{proof}
Using equation (\ref{case1f}):
\begin{equation}
\begin{split}
   & F_{m+q} (-1)^{(n-1)m} F_q^{n-1} F_{p-(n-1)m} - F_m (-1)^{(n-1)m} F_q^{n-1} F_{p-(n-1)m+q} \\
 = & (-1)^{(n-1)m} F_q^{n-1} ( F_{m+q}F_{p-nm+m} - F_mF_{p-nm+m+q} ) \\
 = & (-1)^{nm} F_q^n F_{p-nm} \\
\end{split}
\end{equation}
\end{proof}

For identity (\ref{fibsum2}) the following is needed:
\begin{equation}\label{cases}
 F_{p-nm+1} - (-1)^n F_{p-nm-1} =
 \begin{cases}
   F_{p-nm} & \text{if $n$ is even} \\
   L_{p-nm} & \text{if $n$ is odd} \\
 \end{cases}
\end{equation}
\begin{theorem}
 Let $S(n,m,p,q)$ be the right side of (\ref{fibsum2}), then also:
\begin{equation}
 S(n,m,p,q) = L_{m+q} S(n-1,m,p,q) - L_m S(n-1,m,p+q,q)
\end{equation}
\end{theorem}
\begin{proof}
When $n$ is odd and $n-1$ is even, using (\ref{case1l}) and (\ref{cases}):
\begin{equation}
\begin{split}
   & L_{m+q} 5^{\frac{n-1}{2}} F_q^{n-1} F_{p-(n-1)m} - L_m 5^{\frac{n-1}{2}} F_q^{n-1} F_{p-(n-1)m+q} \\
 = & 5^{\frac{n-1}{2}} F_q^{n-1} ( L_{m+q}F_{p-nm+m} - L_mF_{p-nm+m+q} ) \\
 = & (-1)^{m+1} 5^{\frac{n-1}{2}} F_q^n L_{p-nm} \\
\end{split}
\end{equation}
When $n$ is even and $n-1$ is odd, using (\ref{case1ll}) and (\ref{cases}):\\
\begin{equation}
\begin{split}
   & L_{m+q} 5^{\frac{n-2}{2}} (-1)^{m+1} F_q^{n-1} L_{p-(n-1)m} - L_m 5^{\frac{n-2}{2}} (-1)^{m+1} F_q^{n-1} L_{p-(n-1)m+q} \\
 = & 5^{\frac{n}{2}-1} (-1)^{m+1} F_q^{n-1} ( L_{m+q}L_{p-nm+m} - L_mL_{p-nm+m+q} ) \\
 = & 5^{\frac{n}{2}} F_q^n F_{p-nm} \\
\end{split}
\end{equation}
\end{proof}

Having proved the two idenitities for $G=F$, the following is used:
\begin{equation}
 G_n = \frac{1}{2} [ ( G_{-1}+G_1 ) F_n + G_0 L_n ]
\end{equation}
which is easily proved by checking it for $n=0$ and $n=1$ and $G_{n+2}=G_{n+1}+G_n$.
This means that $G_n$ is a linear combination of $F_n$ and $L_n$,
so it only needs to be checked that the identities hold for $G=F$ and $G=L$.
This is demonstrated by adding the identities for $p-1$ and $p+1$ and
using $F_{p-1}+F_{p+1}=L_p$.

\section{Some Examples}

The identities (\ref{fibsum1}) to (\ref{fibsum4}) with $n=1$ to $n=4$ result in
the following identities.\\
Let $F_n$ be the Fibonacci number, $L_n$ be the Lucas number
and $G_n$ be the generalized Fibonacci number, for which $G_{n+2}=G_{n+1}+G_n$ with
any seed $G_0$ and $G_1$ \cite{BQ03,CW00,V89}.\\
For integer $n$, $m$, $p$:
\begin{equation}\label{final1}
 F_{m+p} G_n - F_m G_{n+p} = (-1)^m F_p G_{n-m}
\end{equation}
\begin{equation}\label{final2}
 L_{m+p} G_n - L_m G_{n+p} = (-1)^{m+1} F_p ( G_{n-m+1} + G_{n-m-1} )
\end{equation}
\begin{equation}\label{final3}
 F_{m+p} G_n - (-1)^p F_m G_{n-p} = F_p G_{n+m}
\end{equation}
\begin{equation}\label{final4}
 L_{m+p} G_n - (-1)^p L_m G_{n-p} = F_p ( G_{n+m+1} + G_{n+m-1} )
\end{equation}
Identities (\ref{final1}) and (\ref{final2}) with $G=F$ and $G=L$ are equivalent
to identities (19a), (19b), (20a) and (20b) in \cite{V89},
and the well known identity \cite{BQ03,CW00,GKP94,V89}:
\begin{equation}
 F_{m+1}F_n + F_mF_{n-1} = F_{n+m}
\end{equation}
is identity (\ref{final3}) with $G=F$ and $p=1$.\\
For integer $n$, $m$, $p$:
\begin{equation}
 F_{m+p}^2 G_n - 2 F_m F_{m+p} G_{n+p} + F_m^2 G_{n+2p} = F_p^2 G_{n-2m}
\end{equation}
\begin{equation}
 L_{m+p}^2 G_n - 2 L_m L_{m+p} G_{n+p} + L_m^2 G_{n+2p} = 5 F_p^2 ( G_{n-2m+1} - G_{n-2m-1} )
\end{equation}
\begin{equation}
 F_{m+p}^2 G_n - 2 (-1)^p F_m F_{m+p} G_{n-p} + F_m^2 G_{n-2p} = F_p^2 G_{n+2m}
\end{equation}
\begin{equation}
 L_{m+p}^2 G_n - 2 (-1)^p L_m L_{m+p} G_{n-p} + L_m^2 G_{n-2p} = 5 F_p^2 ( G_{n+2m+1} - G_{n+2m-1} )
\end{equation}
For integer $n$, $m$, $p$:
\begin{equation}
 F_{m+p}^3 G_n - 3 F_m F_{m+p}^2 G_{n+p} + 3 F_m^2 F_{m+p} G_{n+2p} - F_m^3 G_{n+3p} = (-1)^m F_p^3 G_{n-3m}
\end{equation}
\begin{equation}
\begin{split}
 & L_{m+p}^3 G_n - 3 L_m L_{m+p}^2 G_{n+p} + 3 L_m^2 L_{m+p} G_{n+2p} - L_m^3 G_{n+3p} \\
 & = 5 (-1)^{m+1} F_p^3 ( G_{n-3m+1} + G_{n-3m-1} ) \\
\end{split}
\end{equation}
\begin{equation}
\begin{split}
 & F_{m+p}^3 G_n - 3 (-1)^p F_m F_{m+p}^2 G_{n-p} + 3 F_m^2 F_{m+p} G_{n-2p} - (-1)^p F_m^3 G_{n-3p} \\
 & = F_p^3 G_{n+3m} \\
\end{split}
\end{equation}
\begin{equation}
\begin{split}
 & L_{m+p}^3 G_n - 3 (-1)^p L_m L_{m+p}^2 G_{n-p} + 3 L_m^2 L_{m+p} G_{n-2p} - (-1)^p L_m^3 G_{n-3p} \\
 & = 5 F_p^3 ( G_{n+3m+1} + G_{n+3m-1} ) \\
\end{split}
\end{equation}
For integer $n$, $m$, $p$:
\begin{equation}
\begin{split}
 & F_{m+p}^4 G_n - 4 F_m F_{m+p}^3 G_{n+p} + 6 F_m^2 F_{m+p}^2 G_{n+2p} - 4 F_m^3 F_{m+p} G_{n+3p} \\
 & + F_m^4 G_{n+4p} = F_p^4 G_{n-4m} \\
\end{split}
\end{equation}
\begin{equation}
\begin{split}
 & L_{m+p}^4 G_n - 4 L_m L_{m+p}^3 G_{n+p} + 6 L_m^2 L_{m+p}^2 G_{n+2p} - 4 L_m^3 L_{m+p} G_{n+3p} \\ 
 & + L_m^4 G_{n+4p} = 25 F_p^4 ( G_{n-4m+1} - G_{n-4m-1} ) \\
\end{split}
\end{equation}
\begin{equation}
\begin{split}
 & F_{m+p}^4 G_n - 4 (-1)^p F_m F_{m+p}^3 G_{n-p} + 6 F_m^2 F_{m+p}^2 G_{n-2p} \\
 & - 4 (-1)^p F_m^3 F_{m+p} G_{n-3p} + F_m^4 G_{n-4p} = F_p^4 G_{n+4m} \\
\end{split}
\end{equation}
\begin{equation}
\begin{split}
 & L_{m+p}^4 G_n - 4 (-1)^p L_m L_{m+p}^3 G_{n-p} + 6 L_m^2 L_{m+p}^2 G_{n-2p} \\
 & - 4 (-1)^p L_m^3 L_{m+p} G_{n-3p} + L_m^4 G_{n-4p} = 25 F_p^4 ( G_{n+4m+1} - G_{n+4m-1} ) \\
\end{split}
\end{equation}

\pdfbookmark[0]{References}{}

\end{document}